\documentclass[11pt]{amsart}
\usepackage{graphicx}
\usepackage[all]{xy}
\usepackage{amscd}
\usepackage{amssymb}
\usepackage{amsmath}
\usepackage{amsthm}
\usepackage{amsfonts}
\usepackage{graphics}
\usepackage{epsfig,psfrag}
\usepackage[english]{babel}
\usepackage{latexsym}
\usepackage{mathrsfs}

\newtheorem{theorem}{Theorem}[section]
\newtheorem{corollary}[theorem]{Corollary}
\newtheorem{conjecture}[theorem]{Conjecture}
\newtheorem{lemma}[theorem]{Lemma}

\newtheorem*{RHF}{Riemann-Hurwitz Formula}
\newtheorem*{ELT}{Equivariant Loop Theorem}

\theoremstyle{definition}
\newtheorem{example}[theorem]{Example}
\newtheorem{definition}[theorem]{Definition}

\theoremstyle{remark}
\newtheorem{remark}[theorem]{Remark}

\begin{document}

\title[The most symmetric surfaces in the 3-torus]
{The most symmetric surfaces in the 3-torus}

\author{Sheng Bai}
\address{School of Mathematical Sciences, Peking University, Beijing 100871, CHINA}
\email{Bais@math.pku.edu.cn}

\author{Vanessa Robins}
\address{Department of Applied Mathematics,
Research School of Physics and Engineering, 
The Australian National University, Canberra, Australia}
\email{vanessa.robins@anu.edu.au}

\author{Chao Wang}
\address{School of Mathematical Sciences, University of
Science and Technology of China, Hefei 230026, CHINA}
\email{chao\_{}wang\_{}1987@126.com}

\author{Shicheng Wang}
\address{School of Mathematical Sciences, Peking University, Beijing 100871, CHINA}
\email{wangsc@math.pku.edu.cn}




\thanks{The  second author is supported by grant No.11501534 of NSFC and the last author is supported by grant No.11371034 of NSFC}

\begin{abstract}
Suppose an orientation preserving action of a finite group  $G$ on the closed surface $\Sigma_g$  of genus $g>1$ 
extends over the 3-torus $T^3$ for some embedding $\Sigma_g\subset T^3$.
Then $|G|\le 12(g-1)$, and this upper bound $12(g-1)$ can be achieved for $g=n^2+1, 3n^2+1, 2n^3+1, 4n^3+1, 8n^3+1,  n\in \mathbb{Z}_+$. Those surfaces in $T^3$ realizing the maximum symmetries can be either unknotted or  knotted.
Similar problems in non-orientable category  is also discussed.

Connection with minimal surfaces in $T^3$ is addressed and  when the maximum symmetric surfaces above can be  realized  by minimal surfaces
is identified. \end{abstract}
\maketitle
\vspace{-.5cm}
\tableofcontents

\section{Introduction}

Let $\Sigma_g$ be the closed orientable  surface of genus $g>1$, $\Pi_g$ be the closed non-orientable  surface of genus $g>2$, and $T^3$ be the  three dimensional torus (3-torus for short). We consider the following question in the smooth category:

Suppose the action of a finite group  $G$ on a closed surface $S$  can extend over $T^3$.  Then what is the maximum order of the group  and what does the maximum action look like?

A similar problem has been addressed for surfaces embedded in the 3-sphere, $S^3$, which is the simplest compact 3-manifold in the sense that it is a one point compactification of our three space and the universal spherical 3-manifold covering all spherical 3-manifolds.  See \cite{WWZZ1} and \cite{WWZZ2} for surfaces in the orientable category. 
Note that only orientable surfaces $\Sigma_g$ can be embedded in $S^3$.
$T^3$ is another natural and significant 3-manifold for this question. It is the universal compact Euclidean 3-manifold in the sense that it covers all compact Euclidean 3-manifolds; moreover $T^3$ is covered by our 3-space and the preimage of a closed surface $S\subset T^3$ under such a covering can be a triply periodic surface, which is an interesting object in the natural sciences and engineering~\cite{Hyde}. 

\begin{definition}\label{Def of extendable action} Let $G$ be a finite group.
A $G$-action on a closed surface $S$ is extendable over $T^3$ with respect to an embedding $e: S\hookrightarrow T^3$ if $G$ can act on $T^3$ such that $h\circ e=e\circ h$ for any $h \in G$.
\end{definition}


For an embedding of an orientable surface in $T^3$ we can define whether it is unknotted in the same way as the usual definition for knotting in $S^3$. 

\begin{definition}\label{Def of unknotted and knotted}
An embedding $e: \Sigma_g\hookrightarrow T^3$ is unknotted if $e(\Sigma_g)$ splits $T^3$ into two handlebodies. Otherwise it is knotted.
\end{definition}


We assume all orientable manifolds in this note are already oriented.
According to whether the action preserves the orientation of the surface and $T^3$, we can define four classes of maximum orders.

\begin{definition}\label{Def of maximum order}
Let $S$ be either $\Sigma_g$ or $\Pi_g$. Define $E(S)$, $E^+(S)$, $E_+(\Sigma_g)$ and $E^+_+(\Sigma_g)$ as below:

$E(S)$: the maximum order of all extendable $G$-actions on $S$.

$E^+(S)$: the maximum order of all extendable $G$-actions on $S$ which preserve the orientation of $T^3$.

$E_+(\Sigma_g)$: the maximum order of all extendable $G$-actions on $\Sigma_g$ which preserve the orientation of $\Sigma_g$.

$E^+_+(\Sigma_g)$: the maximum order of extendable $G$-actions on $\Sigma_g$, which preserve the orientation of both $T^3$ and $\Sigma_g$.
\end{definition}

 We will prove the following Theorem \ref{upper bound} and Theorem \ref{sharp upper bound}. Theorem \ref{upper bound}  gives the upper bounds of those invariants defined in Definition \ref{Def of maximum order}, and Theorem \ref{sharp upper bound}
 provides infinitely many $g$ to realize each upper bound in  Theorem \ref{upper bound}.
 
\begin{theorem}\label{upper bound} Suppose the genus $g>1$ for $\Sigma_g$ and $g > 2$ for $\Pi_g$.

(1) $E^+_+(\Sigma_g)\leq 12(g-1)$.

(2) $E_+(\Sigma_g)\leq 24(g-1)$, $E^+(\Sigma_g)\leq 24(g-1)$, $E(\Sigma_g)\leq 48(g-1)$.

(3) $E^+(\Pi_g)\leq 12(g-2)$, $E(\Pi_g)\leq 24(g-2)$.
\end{theorem}

If an extendable $G$-action on $S$ realizes one of the above upper bounds, then the corresponding surface $e(S)$ can be thought as a most symmetric surface in $T^3$.

\begin{theorem}\label{sharp upper bound}   Suppose  $n$ is any  positive integer.

(1)  The upper bound of $E^+_+(\Sigma_g)$ in Theorem \ref{upper bound} can be achieved by 
an unknotted embedding for $g=2n^3+1, 4n^3+1, 8n^3+1$; and 
by a knotted embedding for $g=2n^3+1, 4n^3+1, 8n^3+1$, where $6$ does not divide $n$;  and $g=n^2+1, 3n^2+1$.

(2)  The upper bound of $E_+(\Sigma_g)$, $E^+(\Sigma_g)$ and $E(\Sigma_g)$ in Theorem \ref{upper bound}  can be achieved for $g=2n^3+1, 4n^3+1, 8n^3+1$ by an unknotted embedding.

(3) The upper bound of $E^+(\Pi_g)$ and $E(\Pi_g)$ in Theorem \ref{upper bound}  can be achieved for $g=2n^3+2, 8n^3+2$, where $n$ is odd.
\end{theorem}

The proof of Theorem \ref{upper bound},  using the Riemann-Hurwitz Formula and the Equivariant Loop Theorem, is given in \S2. 
The proof of Theorem \ref{sharp upper bound} occupies the remaining three sections of the paper,  which is outlined as below:

In \S3 and \S4 we 
will construct various examples to realize those upper bounds in Theorem \ref{sharp upper bound} (1),
\S3 for the unknotted case and \S4 for the knotted case, therefore Theorem \ref{sharp upper bound} (1) follows, up to the verification of the knottiness of those embeddings.
Those examples are explicit in the following sense: The 3-torus is obtained by standard opposite face identification of the cube and we define those surfaces in the cube before the identification.
Indeed, for each $g$ in Theorem \ref{sharp upper bound} (1), we construct all embeddings  $\Sigma_g\subset T^3$ realizing
$E^+_+(\Sigma_g)$ we can image for the moment, and we expect those are all embeddings realizing
$E^+_+(\Sigma_g)$, see a conjecture below.
The constructions in \S3 and \S4 often rely on an understanding of crystallographic space groups.

The knottedness of the examples in \S3 and \S4 involved in Theorem \ref{sharp upper bound} (1), as well as Theorem \ref{sharp upper bound} (2) and (3), could be verified by arguments such as those in \S3 and \S4, and by some topological reasoning; 
but we present a more convenient and interesting way in \S5.
In \S5 we first link the verification of knottedness of the examples in \S3 and \S4 with minimal surfaces in $T^3$, or equivalently, triply periodic minimal surfaces in $R^3$, which itself is an important topic, see \cite{Me} and \cite{FH} for examples.  
Historically, many minimal surfaces in $T^3$ were constructed using the symmetry of $T^3$ and it is known that minimal surfaces in $T^3$ must be unknotted \cite{Me}. 
We identify the examples in \S3 with known triply periodic minimal surfaces (up to isotopy) \cite{Br}, therefore they are unknotted.  
On the other hand, by a simple criterion from covering space theory and our constructions, the examples in \S4
are knotted, therefore can not be realized by minimal surfaces.
We also identify the group actions in \S 3 and \S 4 with known space groups as well as their index-2 supergroups; then Theorem \ref{sharp upper bound}  (2) and (3) are proved.

\begin{corollary}
$\Sigma_g \subset T^3$ realizing the upper bound $E^+_+(g)=12(g-1)$ can be realized by minimal surfaces for $g=2n^3+1, 4n^3+1, 8n^3+1$.
\end{corollary}

For each $g$ it is easy to construct an extendable action of order $4(g-1)$ on $(T^3, \Sigma_g)$, see Example 2.3. We end the introduction by the following

\begin{conjecture}  Suppose the genus $g>1$.

(1)  All $g$ realizing the upper bound  $E^+_+(\Sigma_g)=12(g-1)$ are listed in Theorem \ref{sharp upper bound} (1), and moreover, 
 all examples realizing those $g$ are listed in $\S3$ and $\S4$.

(2) $E^+_+(\Sigma_g)=4(g-1)$  for $g\in Z_+\setminus K$, where 
$K=\{f_i(n)| n\in Z_+\}$ and   $f_i$ runs over finitely many quadratic and cubic functions.
\end{conjecture}

\section{Upper and lower bounds  of extendable finite groups}\label{Sec of Maximum order}

We need the following two important results, see \cite{Hu}, \cite{MY}, to prove Theorem \ref{upper bound}.

\begin{RHF}\label{Thm of RH formula}
$\Sigma_g\rightarrow \Sigma_{g'}$ is a regular branched covering with transformation group $G$. Let $a_1, a_2, \cdots, a_k$ be the branched points in $\Sigma_{g'}$ having indices $q_1\leq q_2\leq \cdots\leq q_k$. Then
$$2-2g=|G|(2-2g'-\sum^k_{i=1}(1-\frac{1}{q_i}))$$
\end{RHF}

\begin{ELT}\label{Thm of EV loop}
Let $M$ be a three manifold with a smooth action of a discrete group $G$. Let $F$ be an equivariant subsurface of $\partial M$. If $F$ is not $\pi_1$-injective with respect to inclusion into $M$, then it admits a $G$-equivariant compression disk.
\end{ELT}


Here a nonempty subset $X$ of $M$ is $G$-equivariant if $h(X)=X$ or $h(X)\cap X=\emptyset$, for any  $h\in G$. 
A disk $D\subset M$ is a compression disk of $F$ if $\partial D\subset F$ and in $F$ it does not bound any disk.

\begin{proof}[Proof of {\it(1)}] 
 
Suppose there is an extendable $G$-action on $\Sigma_g$ which preserves the orientation of both $T^3$ and $\Sigma_g$. Cutting $T^3$ along $e(\Sigma_g)$ we get a three manifold $M$. Since both $\Sigma_g$ and $T^3$ are orientable, $\Sigma_g$ must be two-sided in $T^3$, therefore $M$ contains two boundaries $F_1$ and $F_2$, and $F_1\cong F_2\cong \Sigma_g$.

Now clearly $G$ acts on $M$. Since  the $G$-action preserves the orientation of both $T^3$ and $\Sigma_g$, each of  $F_1$ and $F_2$ is a $G$-equivariant surface. Since $g>1$, $\pi_1(\Sigma_g)$ is not abelian, but $\pi_1(T^3)$ is abelian so the induced homomorphism $\pi_1(\Sigma_g)\rightarrow\pi_1(T^3)$ is not injective. 
Then at least one of $F_1$ and $F_2$ is not $\pi_1$-injective with respect to inclusion into $M$. 
Suppose $F_1$ is not $\pi_1$-injective, then it admits a $G$-equivariant compression disk $D$ of $F_1$ by the Equivariant Loop Theorem.

If there exists an $h\in G$ such that $h(D)=D$ and $h$ reverses an orientation of $D$, 
then we can choose a $G$-equivariant regular neighbourhood $N(D)$ of $D$ such that there is a homeomorphism $i: D\times [-1, 1]\rightarrow N(D)$ and $i(D\times \{0\})=D$. Then $D'=i(D\times \{1\})$ is also a $G$-equivariant compression disk of $F_1$, and clearly for this $D'$,  if $h'(D')=D'$ for some $h'\in G$, then $h'$ preserves an orientation of $D'$. 
Hence we can assume each element of $G$ that preserves $D$ also preserves its orientation,
in particular it is fixed point free on $\partial D$.  

Since the $G$-action on $T^3$ preserves the orientations of both $\Sigma_g$ and $T^3$, the induced $G$-action on $M$ preserves $F_1$. 
With the quotient topology $F_1/G$ is homeomorphic to some $\Sigma_{g'}$, 
and $p: F_1\rightarrow F_1/G$ is a regular branched covering. 
Since the $G$ action is fixed point free on $\partial D$ by previous discussion, $p(\partial D)$ is a simple closed curve in $F_1/G$.

Let $a_1, a_2, \cdots, a_k$ be the branch points having indices $q_1\leq q_2\leq \cdots\leq q_k$. Note $2-2g<0$. By the Riemann-Hurwitz Formula we have

$$2-2g'-\sum^k_{i=1}(1-\frac{1}{q_i})=\frac{2-2g}{|G|}<0.$$

If $g'=0$ and $k\leq 3$, then $p(\partial D)$ must bound a disk in $F_1/G$ containing at most one branch point, and $\partial D$ will bound a disk in $F_1$, a contradiction.  Hence either $g'\geq 1$ or $g'=0$ and $k\geq 4$. 

Notice that for each $q_i$, we have 

$$1-\frac{1}{q_i}\geq 1/2.$$ 

Then by elementary calculation we have    $|G|\leq 12(g-1)$, and moreover the equality holds if and only if $g'=0$, $k=4$ and $(q_1, q_2, q_3, q_4)=(2, 2, 2, 3)$.
\end{proof}

\begin{remark} Zimmermann first proved the order of orientation preserving finite group action
on handlebody of genus $g$ is bounded by $12(g-1)$ \cite{Zi} soon after the work \cite{MY}.
\end{remark}

By the above proof, we actually have the following:

\begin{theorem}\label{Thm of 12(g-1)}
Suppose $\Sigma_g$ is embedded in a three manifold $M$ and a finite group $G$  acts on $(\Sigma_g, M)$. 
If $G$ preserves both the two sides and the orientation of $\Sigma_g$ and $\Sigma_g$ is not $\pi_1$-injective in $M$, then $|G|\leq 12(g-1)$.
\end{theorem}

Now {\it(2)} and {\it(3)} follow from  {\it(1)}.

\begin{proof}[Proof of {\it(2)}]
Suppose there is an extendable $G$-action on $\Sigma_g$. Let $G^o$ be the normal subgroup of $G$ containing all elements that preserve both the orientations of $\Sigma_g$ and $T^3$. In the case of $E_+(\Sigma_g)$ or $E^+(\Sigma_g)$, the index of $G^o$ in $G$ is at most two, and in the case of $E(\Sigma_g)$, the index of $G^o$ in $G$ is at most four. Hence $E_+(\Sigma_g)\leq 24(g-1)$, $E^+(\Sigma_g)\leq 24(g-1)$, $E(\Sigma_g)\leq 48(g-1)$.
\end{proof}

\begin{proof}[Proof of {\it(3)}]
Suppose there is an extendable $G$-action on $\Pi_g$ and the action preserves the orientation of $T^3$. We can choose an equivariant regular neighbourhood $N(\Pi_g)$ of $\Pi_g$, then $G$ also acts on $\partial N(\Pi_g)$. Since $T^3$ is orientable, $N(\Pi_g)$ is homeomorphic to a twisted $[-1, 1]$-bundle over $\Pi_g$ and $\partial N(\Pi_g)$ 
is the orientable double cover of $\Pi_g$ under the bundle projection. Since $\chi(\Pi_g)=2-g$, $\chi(\partial N(\Pi_g))=4-2g=2-2(g-1)$.
It follows that  $\partial N(\Pi_g)$
is homeomorphic to $\Sigma_{g-1}$. Clearly the $G$-action preserves the two sides of $\partial N(\Pi_g)$. Since the action preserves the orientation of $T^3$, it also preserves the orientation of $\partial N(\Pi_g)$. Then by the result of {\it(1)}, we have $E^+(\Pi_g)\leq 12(g-2)$. And similar to the proof of {\it(2)}, we have $E(\Pi_g)\leq 24(g-2)$.
\end{proof}

\begin{example}\label{Ex of 4(g-1)} Let $([0,1]^3, +)$ be the unit cube with a cross properly embedded shown as  the right side of Figure 1.
 For each $g>1$,   put $g-1$ copies of $([0,1]^3, +)$ to get $([0,1]^3, +)_{g-1}$ which is a cuboid with an antennae properly embedded, shown as Figure 1.
 If we identify the opposite faces of the  cuboid, we get $(T^3, \Theta_g)$, where $\Theta_g$ is a graph of genus $g$.
 If we first only identify the right and left faces, we get $([0,1]^2\times S^1, A_{g-1})$,  where $A_{g-1}$ is a circular antennae with $g-1$ bars.
 It is easy to see that $G=D_{g-1}\oplus Z_2$ acts on the pair $([0,1]^2\times S^1, A_{g-1})$, where $D_{g-1}$ is the dihedral group acts in a 
 standard way on $([0,1]^2\times S^1, A_{g-1})$,  and $Z_2$ is $\pi$-rotation of $([0,1]^2\times S^1, A_{g-1})$ around the center-circle.
 Clear this action preserves the each pair of faces of $([0,1]^2\times S^1, A_{g-1})$ to be identified, therefore induces an action on $(T^3, \Theta_g)$.
 Let $N(\Theta_g)$ be a $G$-invariant regular neiborhood of $\Theta_g$.  Then $\partial N(\Theta_g)=\Sigma_g$ and $G$ of order $4(g-1)$
 acts on  $(T^3, \Sigma_g)$. 
\end{example}

\begin{figure}[h]
\centerline{\scalebox{0.5}{\includegraphics{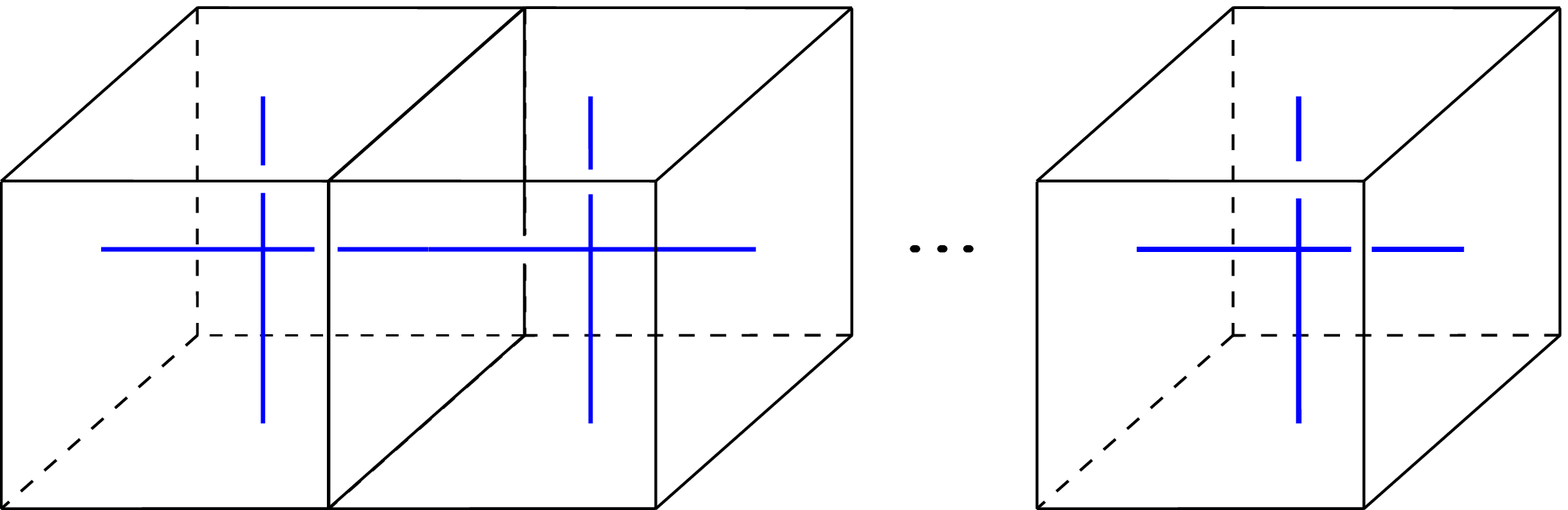}}}
\caption{}
\end{figure}

\section{Unknotted examples of the most symmetric surfaces}\label{Sec of UK example}
In this section we give three classes of examples realizing the upper bound of $E^+_+(\Sigma_g)=12(g-1)$. In each case we first construct a triply periodic graph $\Gamma$ in the three-dimensional Euclidean space $E^3$.  There will be a three-dimensional space group $\mathcal{G}$ preserving $\Gamma$.  Then we choose a rank-three translation normal subgroup $T$ in $\mathcal{G}$.  The space $E^3/T$ is our $T^3$. The finite group $G=\mathcal{G}/T$ acts on $T^3$ preserving the graph $\Theta=\Gamma/T$. Finally, we choose an equivariant regular neighbourhood $N(\Theta)$ of $\Theta$, and $\partial N(\Theta)$ is our surface $\Sigma_g$.  
After the discussion of \S5 we will see that for each example the complement of $N(\Theta)$ 
is also a handlebody, hence the surface is unknotted.

\begin{definition}\label{Def of transgroup}
Let $T_1=\{(a, b, c)\mid a, b, c\in \mathbb{Z}\}$ be the group of integer translations. Let $t_x=(1, 0, 0), t_y=(0, 1, 0), t_z=(0, 0, 1)$. An element $t=(a, b, c)\in T_1$ acts on $E^3$ as following:
$$t: (x, y, z)\mapsto(x+a, y+b, z+c)$$
For $n\in \mathbb{Z}_+$, we define three classes of subgroups $T_m$ of $T_1$ for those integers $m$ that can be presented (uniquely) in the form $n^3$, $2n^3$ and $4n^3$ as follows:
\begin{align*}
T_{n^3}&=\langle nt_x, nt_y, nt_z\rangle,\\
T_{2n^3}&=\langle nt_y+nt_z, nt_z+nt_x, nt_x+nt_y\rangle,\\
T_{4n^3}&=\langle -nt_x+nt_y+nt_z, nt_x-nt_y+nt_z, nt_x+nt_y-nt_z\rangle.
\end{align*}
Here the subscript $m$ of $T_m$ is equal to the volume $Vol(E^3/T_m)$, $m=n^3, 2n^3, 4n^3$. 
\end{definition}
Clearly $T_{2n^3}$ and  $T_{4n^3}$ are subgroups of $T_{n^3}$.
One can easily verify that $T_{(2n)^3}\subset T_{2n^3}$, since $2nt_x, 2nt_y, 2nt_z$ are linear combinations of 
$nt_y+nt_z, nt_z+nt_x, nt_x+nt_y$,
hence $T_{32n^3}=T_{4(2n)^3}\subset T_{(2n)^3}=T_{8n^3} \subset T_{2n^3}$.
Similarly one can  verify that $T_{(2n)^3} \subset T_{4n^3}$, hence 
 $T_{16n^3} = T_{2(2n)^3} \subset T_{(2n)^3} = T_{8n^3} \subset T_{4n^3}$.


\begin{definition}\label{Def of iso of E3}
We define five isometries of $E^3$ as follows:
\begin{align*}
r_y: (x, y, z)&\mapsto(-x, y, -z)\\
r_z: (x, y, z)&\mapsto(-x, -y, z)\\
r_{xy}: (x, y, z)&\mapsto(y, x, -z)\\
r_{xyz}: (x, y, z)&\mapsto(z, x, y)\\
t_{1/2}: (x, y, z)&\mapsto(x+1/2, y+1/2, z+1/2)
\end{align*}
\end{definition}

The isometries $r_y$, $r_z$ and $r_{xy}$ are 2-fold rotations  (i.e., rotation by an angle of $\pi$) about the $y$-axis, $z$-axis and the line $x=y$, $z=0$ respectively. 
The isometry $r_{xyz}$ is a positive $3$-fold rotation about the cube body diagonal, $x=y=z$ (i.e., right-hand rule rotation by $2\pi/3$ about the direction $[1,1,1]$.) 

The notations of space groups used below come from \cite{Ha}.

\begin{example}\label{Ex of scP}
Let $\Gamma^1_{min}$ be the one skeleton of the unit cube $[0, 1]^3$ in $E^3$ as in Figure \ref{fig:ExScPmin}. Let $\Gamma^1=\bigcup_{t\in T_1}t(\Gamma^1_{min})$. Then $\Gamma^1$, the 1-skeleton of the tessellation of $E^3$ by the unit cube, is a triply periodic graph called the simple or primitive cubic lattice (\textbf{pcu} in \cite{rcsr}).

\begin{figure}[h]
\centerline{\scalebox{0.4}{\includegraphics{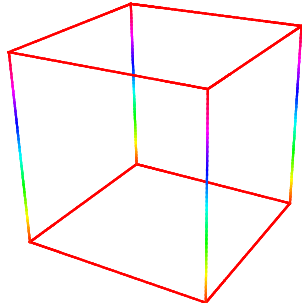}}}
\caption{One skeleton of $[0,1]^3$}\label{fig:ExScPmin}
\end{figure}

Let $H^1=\langle r_y, r_z, r_{xy}, r_{xyz}\rangle$, which is the well-known orientation-preserving isometric group of $[-1, 1]^3$ of order 24 
(with Sch\"onflies symbol $O$). 
Let $\mathcal{G}^1=\langle T_1, H^1\rangle$, this is the space group $P432$. 
Then $\mathcal{G}^1$ preserves $\Gamma^1$. 
Now in $T^3\cong E^3/T_1$ we have a graph $\Theta^1_1=\Gamma^1/T_1$. It has one vertex and three edges, hence its Euler characteristic $\chi(\Theta^1_1)=-2$ and its genus $g=1-\chi(\Theta^1_1)=3$. Clearly $G^1_1=\mathcal{G}^1/T_1\cong H^1$ acts on $T^3\cong E^3/T_1$ preserving $\Theta^1_1$, and $G^1_1$ has order $24$. Hence when we choose an equivariant regular neighbourhood $N(\Theta^1_1)$ of $\Theta^1_1$, we get an extendable action of order $24$ on $\Sigma_3\cong\partial N(\Theta^1_1)$.

Similarly in $T^3\cong E^3/T_{n^3}$ we have a graph $\Theta^1_{n^3}=\Gamma^1/T_{n^3}$. Since $T_{n^3}$ is a normal subgroup of $\mathcal{G}^1$, $G^1_{n^3}=\mathcal{G}^1/T_{n^3}$ acts on $T^3\cong E^3/T_{n^3}$ preserving $\Theta^1_{n^3}$.
\begin{align*}
\chi(\Theta^1_{n^3})&=\chi(\Theta^1_1)\cdot Vol(E^3/T_{n^3})=-2n^3\\
|G^1_{n^3}|&=|G^1_1|\cdot Vol(E^3/T_{n^3})=24n^3
\end{align*}
Hence the genus of $\Theta^1_{n^3}$ is $2n^3+1$. Then we can get an extendable action of order $24n^3$ on $\Sigma_{2n^3+1}\cong\partial N(\Theta^1_{n^3})$.

Notice that when $m=2n^3, 4n^3, n\in \mathbb{Z}_+$, $T_m$ is also a normal subgroup of $\mathcal{G}^1$. Similar to the above construction, we can get an extendable action of order $24m$ on $\Sigma_{2m+1}\cong\partial N(\Theta^1_m)$. Here $\Theta^1_m=\Gamma^1/T_m$ is a graph in $T^3\cong E^3/T_m$.
\end{example}

In the above example the superscript $1$ in $\Gamma^1$ or $\Theta^1_m$ is equal to the volume of the `minimal 3-torus' $E^3/T_1$. It is also equal to the number of vertices of the graph $\Theta^1_1$. In following examples we use similar notations.

\begin{example}\label{Ex of scD}
Let $\Gamma^2_{min}$ be the graph in the unit cube $[0, 1]^3$ in $E^3$ as in Figure \ref{fig:ExScDmin}. It consists of four edges from $(1/2, 1/2, 1/2)$ to $(0, 0, 0)$, $(0, 1, 1)$, $(1, 0, 1)$, $(1, 1, 0)$. Let $\Gamma^2=\bigcup_{t\in T_2}t(\Gamma^2_{min})$. One can check that it is connected and is known to scientists as the bonding structure of diamond (\textbf{dia} in \cite{rcsr}).  Figure \ref{fig:ExScDmin} shows a fundamental region of $T_2$.

\begin{figure}[h]
\centerline{\scalebox{0.6}{\includegraphics{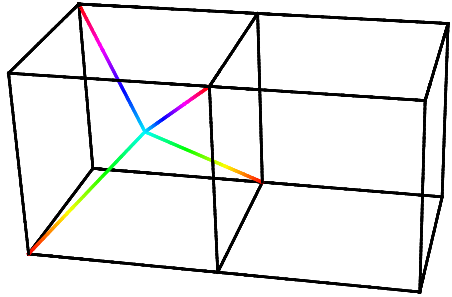}}}
\caption{$\Gamma^2_{min}$ in $[0,2]\times[0,1]\times[0,1]$}\label{fig:ExScDmin}
\end{figure}

Let $H^2=\langle r_y, r_z, r_{xyz}\rangle$. It is the orientation-preserving isometric group of the regular tetrahedron formed by the convex hull of $(1, 1, 1)$, $(1, -1, -1)$, $(-1, 1, -1)$, and $(-1, -1, 1)$ and has Sch\"onflies symbol $T$.  Let $\mathcal{G}^2=\langle T_2, H^2, t_{1/2}r_{xy}\rangle$, this is the space group $F4_1 32$.  Then $\mathcal{G}^2$ preserves $\Gamma^2$.

Now in $T^3\cong E^3/T_2$ we have a graph $\Theta^2_2=\Gamma^2/T_2$. It has two vertices and four edges, hence its Euler characteristic $\chi(\Theta^2_2)=-2$ and its genus is $3$. $T_2$ is a normal subgroup of $\mathcal{G}^2$ and $G^2_2=\mathcal{G}^2/T_2$ has order $24$. It acts on $T^3\cong E^3/T_2$ preserving $\Theta^2_2$. Hence when we choose an equivariant regular neighbourhood $N(\Theta^2_2)$ of $\Theta^2_2$, we get an extendable action of order $24$ on $\Sigma_3\cong\partial N(\Theta^2_2)$.

Similarly when $m=n^3, 4n^3, 16n^3, n\in \mathbb{Z}_+$, in $T^3\cong E^3/T_{2m}$ we have a graph $\Theta^2_{2m}=\Gamma^2/T_{2m}$. One can check that $T_{2m}$ is a normal subgroup of $\mathcal{G}^2$, and the quotient $G^2_{2m}=\mathcal{G}^2/T_{2m}$ acts on $T^3\cong E^3/T_{2m}$ preserving $\Theta^2_{2m}$.
\begin{align*}
\chi(\Theta^2_{2m})&=\chi(\Theta^2_2)\cdot Vol(E^3/T_{2m})/Vol(E^3/T_2)=-2m\\
|G^2_{2m}|&=|G^2_2|\cdot Vol(E^3/T_{2m})/Vol(E^3/T_2)=24m
\end{align*}
Hence the genus of $\Theta^2_{2m}$ is $2m+1$. Then we can get an extendable action of order $24m$ on $\Sigma_{2m+1}\cong\partial N(\Theta^2_{2m})$.
\end{example}

\begin{example}\label{Ex of shG}
Let $\gamma$ be the graph in the unit cube $[0, 1]^3$ in $E^3$ as in Figure \ref{fig:ExShGmin}. It has three edges from $(1/4, 1/4, 1/4)$ to $(0, 1/2, 1/4)$, $(1/4, 0, 1/2)$ and $(1/2, 1/4, 0)$. Let $\Gamma^4_{min}=\gamma\cup t^2_yt_zr_yr_z(\gamma)\cup t^2_xt_yt_zr_y(\gamma)\cup t^2_xt_yr_z(\gamma)$. Then let $\Gamma^4=\bigcup_{t\in T_4}t(\Gamma^4_{min})$.  One can check that it is connected, and is the oft-rediscovered chiral vertex-transitive net of degree-3 known by many names~\cite{HOP}, including $\textbf{srs}$ in \cite{rcsr}. 
Figure \ref{fig:ExShGmin} shows a fundamental region of $T_4$.

\begin{figure}[h]
\centerline{\scalebox{0.6}{\includegraphics{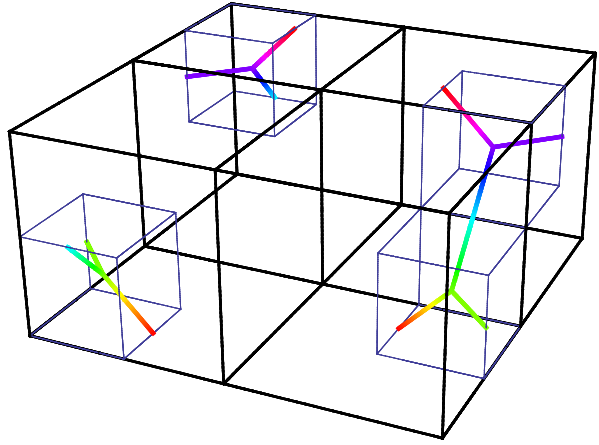}}}
\caption{$\Gamma^4_{min}$ in $[0,2]\times[0,2]\times[0,1]$}\label{fig:ExShGmin}
\end{figure}

Let $\mathcal{G}^4=\langle T_4, t_xt_zr_z, t_yt_zr_y, r_{xyz}, t_xt_{1/2}r_{xy}\rangle$, it is the space group $I4_1 32$.  Then $\mathcal{G}^4$ preserves $\Gamma^4$.

Now in $T^3\cong E^3/T_4$ we have a graph $\Theta^4_4=\Gamma^4/T_4$. It has four vertices and six edges, hence its Euler characteristic $\chi(\Theta^4_4)=-2$ and its genus is $3$. $T_4$ is a normal subgroup of $\mathcal{G}^4$ and $G^4_4=\mathcal{G}^4/T_4$ has order $24$. It acts on $T^3\cong E^3/T_4$ preserving $\Theta^4_4$. Hence when we choose an equivariant regular neighbourhood $N(\Theta^4_4)$ of $\Theta^4_4$, we get an extendable action of order $24$ on $\Sigma_3\cong\partial N(\Theta^4_4)$.

Similarly when $m=n^3, 2n^3, 4n^3, n\in \mathbb{Z}_+$, in $T^3\cong E^3/T_{4m}$ we have a graph $\Theta^4_{4m}=\Gamma^4/T_{4m}$. One can check that $T_{4m}$ is a normal subgroup of $\mathcal{G}^4$, and the quotient $G^4_{4m}=\mathcal{G}^4/T_{4m}$ acts on $T^3\cong E^3/T_{4m}$ preserving $\Theta^4_{4m}$.
\begin{align*}
\chi(\Theta^4_{4m})&=\chi(\Theta^4_4)\cdot Vol(E^3/T_{4m})/Vol(E^3/T_4)=-2m\\
|G^4_{4m}|&=|G^4_4|\cdot Vol(E^3/T_{4m})/Vol(E^3/T_4)=24m
\end{align*}
Hence the genus of $\Theta^4_{4m}$ is $2m+1$. Then we can get an extendable action of order $24m$ on $\Sigma_{2m+1}\cong\partial N(\Theta^4_{4m})$.
\end{example}

\section{Knotted examples of the most symmetric surfaces}\label{Sec of K example}

In the section we give a further six classes of examples realizing the upper bound of $E^+_+(\Sigma_g)=12(g-1)$. In each case below the embedding is knotted. The surface $e(\Sigma_g)$ does bound a handlebody in $T^3$ on one side, but the complement of the handlebody is not a handlebody, and this will be clear after the discussion of \S5.

Constructions of the extendable actions and surfaces is similar to \S4, but here the graph $\Gamma$ in $E^3$ can be disconnected. There is a space group $\mathcal{G}$ preserving $\Gamma$ and a rank three translation normal subgroup $T$ in $\mathcal{G}$ so that the graph $\Theta=\Gamma/T$ is connected. The finite group $G=\mathcal{G}/T$ acts on $T^3 \cong E^3/T$ preserving $\Theta$.  Then we choose an equivariant regular neighbourhood $N(\Theta)$ of $\Theta$, and $\partial N(\Theta)$ is our surface $\Sigma_g$.

\begin{example}\label{Ex of KS from scP}
Let $\Gamma^{1,2}=\Gamma^1\cup t_{1/2}(\Gamma^1)$. The graph $t_{1/2}(\Gamma^1)$ can be thought of as the dual of $\Gamma^1$ in $E^3$. $\Gamma^{1,2}$ has two connected components and is named \textbf{pcu-c} in \cite{rcsr}.  In the unit cube $[0, 1]^3$ it is as in Figure \ref{fig:ExScPdual}.

\begin{figure}[h]
\centerline{\scalebox{0.4}{\includegraphics{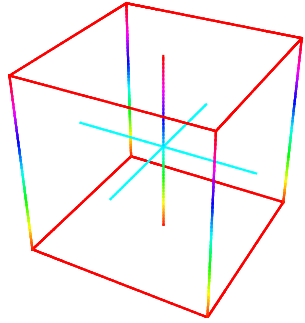}}}
\caption{Part of $\Gamma^{1,2}$ in $[0,1]^3$}\label{fig:ExScPdual}
\end{figure}

Let $\mathcal{G}^{1,2}=\langle \mathcal{G}^1, t_{1/2}\rangle$, it is the space group $I432$, and let 
$T_{1/2}=\langle T_1, t_{1/2}\rangle=\{(a/2, b/2, c/2)\mid (a, b, c)\in T_4\}$. 
Then $\mathcal{G}^{1,2}$ preserves $\Gamma^{1,2}$ and $T_{1/2}$ is a normal subgroup of $\mathcal{G}^{1,2}$. 
Since $t_{1/2}$ changes the two components of $\Gamma^{1,2}$, the graph $\Theta^{1,2}_{1/2}=\Gamma^{1,2}/T_{1/2}$ in $T^3\cong E^3/T_{1/2}$ is connected. Its Euler characteristic $\chi(\Theta^{1,2}_{1/2})=-2$ and its genus is $3$. $G^{1,2}_{1/2}=\mathcal{G}^{1,2}/T_{1/2}$ has order $24$. It acts on $T^3\cong E^3/T_{1/2}$ preserving $\Theta^{1,2}_{1/2}$. Hence when we choose an equivariant regular neighbourhood $N(\Theta^{1,2}_{1/2})$ of $\Theta^{1,2}_{1/2}$, we get an extendable action of order $24$ on $\Sigma_3\cong\partial N(\Theta^{1,2}_{1/2})$.

Similarly let $T_{n^3/2}=\{(a/2, b/2, c/2)\mid (a, b, c)\in T_{4n^3}\}, n\in \mathbb{Z}_+, 2\nmid n$. Then the graph $\Theta^{1,2}_{n^3/2}=\Gamma^{1,2}/T_{n^3/2}$ in $T^3\cong E^3/T_{n^3/2}$ is connected. Since $T_{n^3/2}$ is a normal subgroup of $\mathcal{G}^{1,2}$, $G^{1,2}_{n^3/2}=\mathcal{G}^{1,2}/T_{n^3/2}$ acts on $T^3\cong E^3/T_{n^3/2}$ preserving $\Theta^{1,2}_{n^3/2}$.
\begin{align*}
\chi(\Theta^{1,2}_{n^3/2})&=\chi(\Theta^{1,2}_{1/2})\cdot Vol(E^3/T_{n^3/2})/Vol(E^3/T_{1/2})=-2n^3\\
|G^{1,2}_{n^3/2}|&=|G^{1,2}_{1/2}|\cdot Vol(E^3/T_{n^3/2})/Vol(E^3/T_{1/2})=24n^3
\end{align*}
Hence the genus of $\Theta^{1,2}_{n^3/2}$ is $2n^3+1$. Then we can get an extendable action of order $24n^3$ on $\Sigma_{2n^3+1}\cong\partial N(\Theta^{1,2}_{n^3/2})$, here $n\in \mathbb{Z}_+, 2\nmid n$.
\end{example}

\begin{example}\label{Ex of KS from scD}
Let $\Gamma^{2,2}=\Gamma^2\cup t_x(\Gamma^2)$. The graph $t_x(\Gamma^2)$ can be thought as the dual of $\Gamma^2$ in $E^3$. $\Gamma^{2,2}$ has two connected components and is called \textbf{dia-c} in \cite{rcsr}. 
In the fundamental region of $T_2$ it is as in Figure \ref{fig:ExScDdual}.

\begin{figure}[h]
\centerline{\scalebox{0.6}{\includegraphics{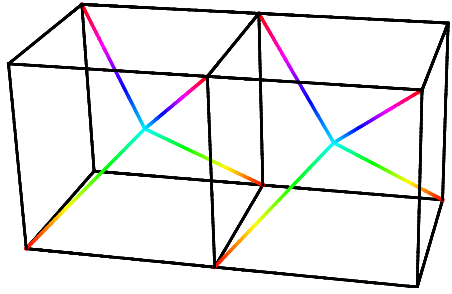}}}
\caption{Part of $\Gamma^{2,2}$ in $[0,2]\times[0,1]\times[0,1]$}\label{fig:ExScDdual}
\end{figure}

Let $\mathcal{G}^{2,2}=\langle \mathcal{G}^2, t_x\rangle$, it is the space group $P4_232$.  Then $\mathcal{G}^{2,2}$ preserves $\Gamma^{2,2}$ and contains $T_1$ as a normal subgroup. Since $t_x$ changes the two components of $\Gamma^{2,2}$, the graph $\Theta^{2,2}_1=\Gamma^{2,2}/T_1$ in $T^3\cong E^3/T_1$ is connected.  Its Euler characteristic $\chi(\Theta^{2,2}_1)=-2$, and its genus is $3$.   $G^{2,2}_1=\mathcal{G}^{2,2}/T_1$ has order $24$. It acts on $T^3\cong E^3/T_1$ preserving $\Theta^{2,2}_1$.  Hence when we choose an equivariant regular neighbourhood $N(\Theta^{2,2}_1)$ of $\Theta^{2,2}_1$, we get an extendable action of order $24$ on $\Sigma_3\cong\partial N(\Theta^{2,2}_1)$.

Similarly if $m=n^3, 4n^3, n\in \mathbb{Z}_+, 2\nmid n$, the graph $\Theta^{2,2}_m=\Gamma^{2,2}/T_m$ in $T^3\cong E^3/T_m$ is connected. One can check that $T_m$ is a normal subgroup of $\mathcal{G}^{2,2}$, $G^{2,2}_m=\mathcal{G}^{2,2}/T_m$ acts on $T^3\cong E^3/T_m$ preserving $\Theta^{2,2}_m$.
\begin{align*}
\chi(\Theta^{2,2}_m)&=\chi(\Theta^{2,2}_1)\cdot Vol(E^3/T_m)=-2m\\
|G^{2,2}_m|&=|G^{2,2}_1|\cdot Vol(E^3/T_m)=24m
\end{align*}
Hence the genus of $\Theta^{2,2}_m$ is $2m+1$. Then we can get an extendable action of order $24m$ on $\Sigma_{2m+1}\cong\partial N(\Theta^{2,2}_m)$, here $m=n^3, 4n^3, n\in \mathbb{Z}_+, 2\nmid n$.
\end{example}

\begin{example}\label{Ex of KS from scD and shG}
Let $\Gamma^{4,4}=\Gamma^4\cup t_x(\Gamma^4)\cup t_y(\Gamma^4)\cup t_z(\Gamma^4)$. $\Gamma^{4,4}$ has four connected components.  In the unit cube $[0,1]^3$ in $E^3$ it is as in Figure \ref{fig:ExScDShG}. 

\begin{figure}[h]
\centerline{\scalebox{0.6}{\includegraphics{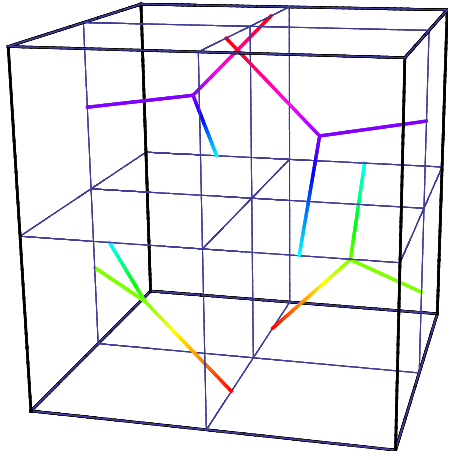}}}
\caption{Part of $\Gamma^{4,4}$ in $[0,1]^3$}\label{fig:ExScDShG}
\end{figure}

One can check that $\mathcal{G}^{2,2}$ defined in Example \ref{Ex of KS from scD} (space group $P4_232$) preserves $\Gamma^{4,4}$, and the graph $\Theta^{4,4}_1=\Gamma^{4,4}/T_1$ in $T^3\cong E^3/T_1$ is connected. Its Euler characteristic $\chi(\Theta^{4,4}_1)=-2$, and its genus is $3$. $G^{2,2}_1=\mathcal{G}^{2,2}/T_1$ has order $24$. It acts on $T^3\cong E^3/T_1$ preserving $\Theta^{4,4}_1$. Hence when we choose an equivariant regular neighbourhood $N(\Theta^{4,4}_1)$ of $\Theta^{4,4}_1$, we get an extendable action of order $24$ on $\Sigma_3\cong\partial N(\Theta^{4,4}_1)$.

Similarly if $m=n^3, 2n^3, n\in \mathbb{Z}_+, 2\nmid n$, the graph $\Theta^{4,4}_m=\Gamma^{4,4}/T_m$ in $T^3\cong E^3/T_m$ is connected. $G^{2,2}_m$ acts on $T^3\cong E^3/T_m$ preserving $\Theta^{4,4}_m$.
\begin{align*}
\chi(\Theta^{4,4}_m)&=\chi(\Theta^{4,4}_1)\cdot Vol(E^3/T_m)=-2m\\
|G^{2,2}_m|&=|G^{2,2}_1|\cdot Vol(E^3/T_m)=24m
\end{align*}
Hence the genus of $\Theta^{4,4}_m$ is $2m+1$. Then we can get an extendable action of order $24m$ on $\Sigma_{2m+1}\cong\partial N(\Theta^{2,2}_m)$, here $m=n^3, 2n^3, n\in \mathbb{Z}_+, 2\nmid n$.
\end{example}

\begin{example}\label{Ex of KS from scP and shG}
Let $\Gamma^{4,8}=\Gamma^{4,4}\cup t_{1/2}(\Gamma^{4,4})$. It has eight connected components. In the unit cube $[0, 1]^3$ it is as in Figure \ref{fig:ExScPShG}.

\begin{figure}[h]
\centerline{\scalebox{0.6}{\includegraphics{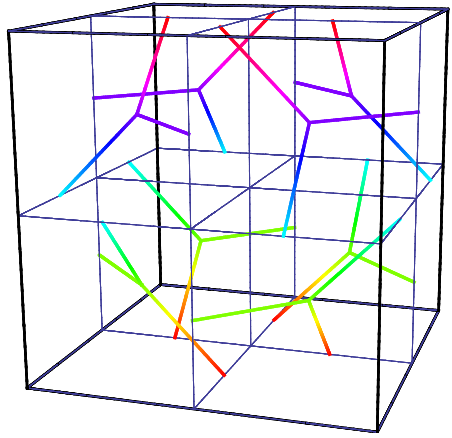}}}
\caption{Part of $\Gamma^{4,8}$ in $[0,1]^3$}\label{fig:ExScPShG}
\end{figure}

One can check that $\mathcal{G}^{1,2}$ defined in Example \ref{Ex of KS from scP} (space group $I432$)  preserves $\Gamma^{4,8}$, and the graph $\Theta^{4,8}_{1/2}=\Gamma^{4,8}/T_{1/2}$ in $T^3\cong E^3/T_{1/2}$ is connected. Its Euler characteristic is $-2$ and its genus is $3$. The order $24$ group $G^{1,2}_{1/2}$ acts on $T^3\cong E^3/T_{1/2}$ preserving $\Theta^{4,8}_{1/2}$. Hence when we choose an equivariant regular neighbourhood $N(\Theta^{4,8}_{1/2})$ of $\Theta^{4,8}_{1/2}$, we get an extendable action of order $24$ on $\Sigma_3\cong\partial N(\Theta^{4,8}_{1/2})$.

Similarly for $n\in \mathbb{Z}_+, 2\nmid n$, the graph $\Theta^{4,8}_{n^3/2}=\Gamma^{4,8}/T_{n^3/2}$ in $T^3\cong E^3/T_{n^3/2}$ is connected. $G^{1,2}_{n^3/2}$ acts on $T^3\cong E^3/T_{n^3/2}$ preserving $\Theta^{4,8}_{n^3/2}$.
\begin{align*}
\chi(\Theta^{4,8}_{n^3/2})&=\chi(\Theta^{4,8}_{1/2})\cdot Vol(E^3/T_{n^3/2})/Vol(E^3/T_{1/2})=-2n^3\\
|G^{1,2}_{n^3/2}|&=|G^{1,2}_{1/2}|\cdot Vol(E^3/T_{n^3/2})/Vol(E^3/T_{1/2})=24n^3
\end{align*}
Hence the genus of $\Theta^{4,8}_{n^3/2}$ is $2n^3+1$. Then we can get an extendable action of order $24n^3$ on $\Sigma_{2n^3+1}\cong\partial N(\Theta^{4,8}_{n^3/2})$, here $n\in \mathbb{Z}_+, 2\nmid n$.
\end{example}

\begin{example}\label{Ex of KG from shG}
Let $\gamma'$ be the graph in the unit cube $[0, 1]^3$ in $E^3$ as in Figure 3. There are three edges from $(1/4, 1/4, 1/4)$ to $(0, 1/4, 1/2)$, $(1/2, 0, 1/4)$, $(1/4, 1/2, 0)$, and three edges from $(0, 3/4, 1/2)$ to $(1/2, 3/4, 1)$, $(1/2, 0, 3/4)$ to $(1, 1/2, 3/4)$ and $(3/4, 1/2, 0)$ to $(3/4, 1, 1/2)$ separately.

Let $\Gamma'^4_{min}=\gamma'\cup t^2_yt_zr_yr_z(\gamma')\cup t^2_xt_yt_zr_y(\gamma')\cup t^2_xt_yr_z(\gamma')$, and $\Gamma'^4=\bigcup_{t\in T_4}t(\Gamma'^4_{min})$. One can check that $\Gamma'^4$ has $27$ connected components, each is similar to a mirror image of $\Gamma^4$.  Figure \ref{fig:ExShGknot} shows the part of $\Gamma'^4$ in a fundamental region of $T_4$.


\begin{figure}[h]
\centerline{\scalebox{0.5}{\includegraphics{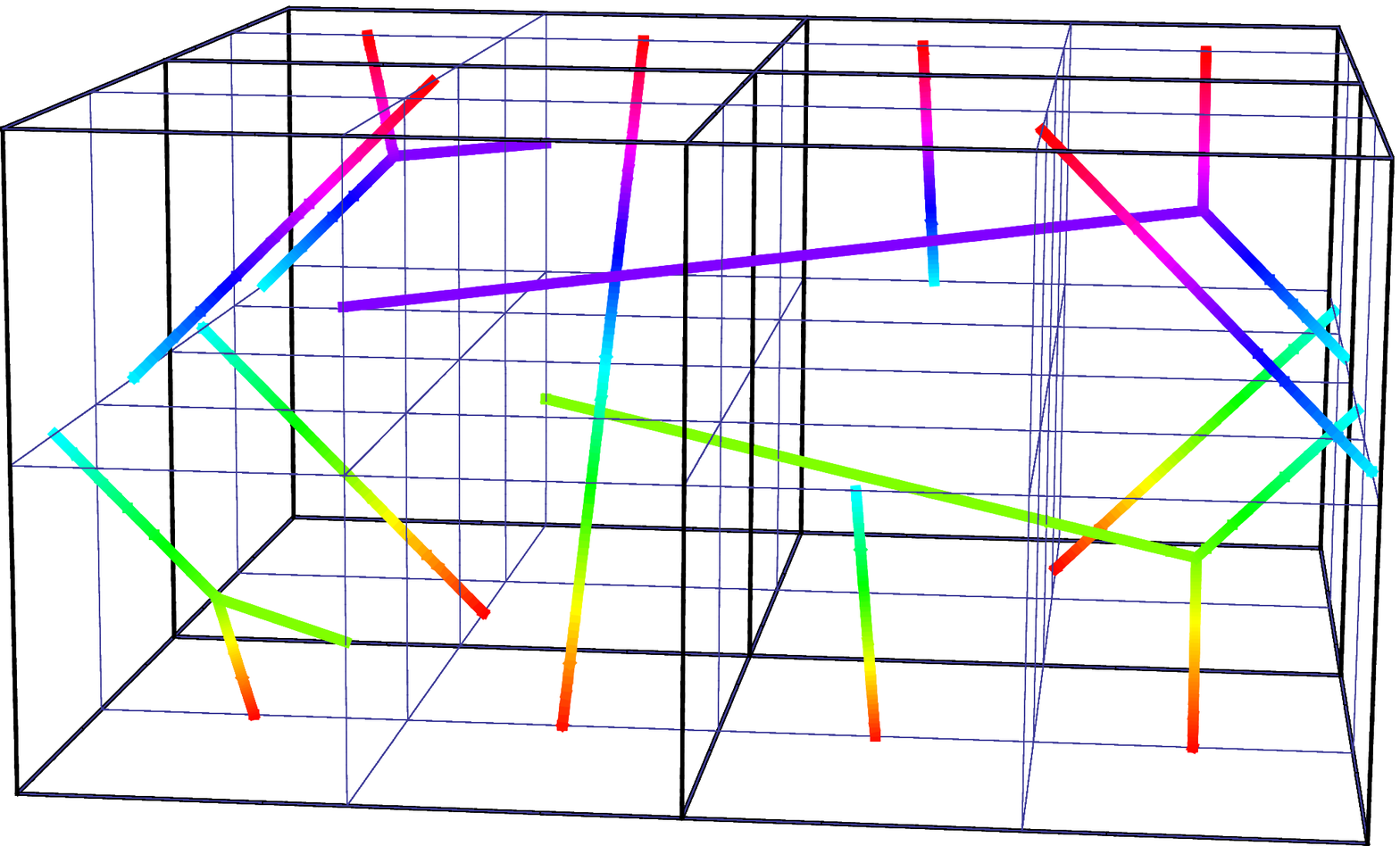}}}
\caption{Part of $\Gamma'^4{min}$ in $[0,2]\times[0,2]\times[0,1]$}\label{fig:ExShGknot}
\end{figure}

By  construction, $\mathcal{G}^4$ (space group $I4_132$) preserves $\Gamma'^4$. And one can check that $\Theta'^4_4=\Gamma'^4/T_4$ in $T^3\cong E^3/T_4$ is connected. $\Theta'^4_4$ has four vertices and six edges, hence it has Euler characteristic $-2$ and genus $3$. 
The order-$24$ group $G^4_4$ acts on $T^3\cong E^3/T_4$ preserving $\Theta'^4_4$. Hence when we choose an equivariant regular neighbourhood $N(\Theta'^4_4)$ of $\Theta'^4_4$, we get an extendable action of order $24$ on $\Sigma_3\cong\partial N(\Theta'^4_4)$.

Similarly when $m=n^3, 2n^3, 4n^3, n\in \mathbb{Z}_+, 3\nmid n$, the graph $\Theta'^4_{4m}=\Gamma'^4/T_{4m}$ in $T^3\cong E^3/T_{4m}$ is connected. The quotient $G^4_{4m}=\mathcal{G}^4/T_{4m}$ acts on $T^3\cong E^3/T_{4m}$ preserving $\Theta'^4_{4m}$.
\begin{align*}
\chi(\Theta'^4_{4m})&=\chi(\Theta'^4_4)\cdot Vol(E^3/T_{4m})/Vol(E^3/T_4)=-2m\\
|G^4_{4m}|&=|G^4_4|\cdot Vol(E^3/T_{4m})/Vol(E^3/T_4)=24m
\end{align*}
Hence the genus of $\Theta'^4_{4m}$ is $2m+1$.  Then we can get an extendable action of order $24m$ on $\Sigma_{2m+1}\cong\partial N(\Theta'^4_{4m})$, here $m=n^3, 2n^3, 4n^3, n\in \mathbb{Z}_+, 3\nmid n$.
\end{example}

\begin{definition}
Let $t_\omega=(-1/2,\sqrt{3}/2,0)$. Define
\begin{align*}
T^\omega_{n^2}&=\langle nt_\omega, nt_x, t_z\rangle,\\
T^\omega_{3n^2}&=\langle 2nt_\omega+nt_x, nt_\omega+2nt_x, t_z\rangle.
\end{align*}
Define a rotation $r_\omega$ on $E^3$ as following:
$$r_\omega: (x, y, z)\mapsto(-\frac{1}{2}x - \frac{\sqrt{3}}{2}y,  \frac{\sqrt{3}}{2}x - \frac{1}{2}y, z)$$
\end{definition}

$T^\omega_1$ is the hexagonal 3D lattice and $r_\omega$ is a rotation by $2\pi/3$ right-handed with respect to the direction $[0,0,1]$. 

\begin{example}\label{Ex of rhombus}
Let $P$ be a regular hexagon in the $xy$-plane. It has center $(0,0,0)$ and contains $(1/2,\sqrt{3}/6,0)$ as a vertex, see Figure  \ref{fig:Hex}. Let $\Gamma^P_{min}$ be the boundary of $P$ in $E^3$. Let $\Gamma^P=\bigcup_{t\in T^\omega_1}t(\Gamma^P_{min})$. It contains infinitely many connected components, and for each $n\in \mathbb{Z}$, the horizontal plane $\{(x,y,n)\mid x,y\in E^2\}$ contains exactly one of them. A local picture of $\Gamma^P$ is as in Figure \ref{fig:ExHexknot}.

\begin{figure}[h]
\centerline{\scalebox{0.6}{\includegraphics{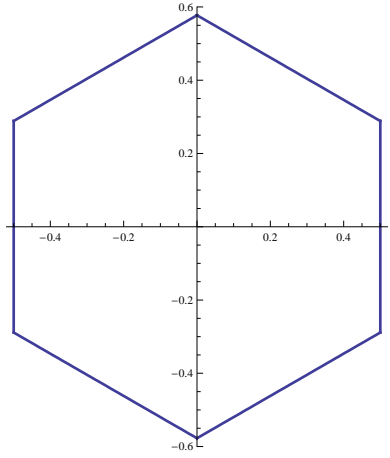}}}
\caption{A regular hexagon}\label{fig:Hex}
\end{figure}

\begin{figure}[h]
\centerline{\scalebox{0.6}{\includegraphics{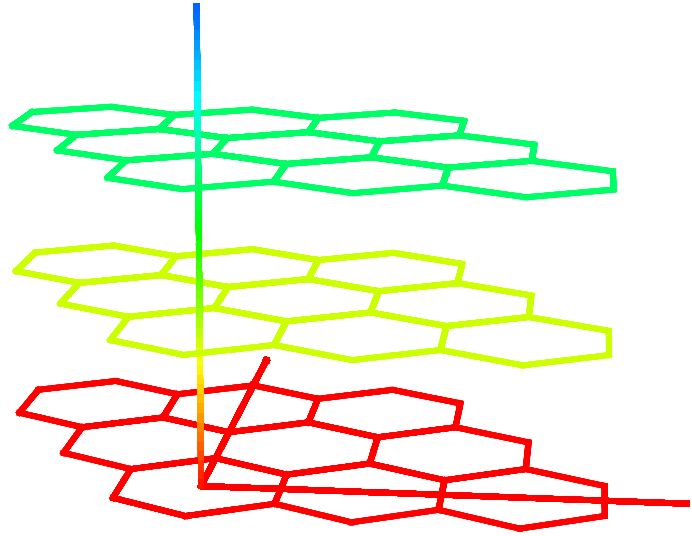}}}
\caption{A local picture of $\Gamma^P$}\label{fig:ExHexknot}
\end{figure}

Let $H^\omega=\langle r_\omega, r_x, r_y\rangle$; it is the isometric group of $P$ (Sch\"onflies symbol $D_6$). Let $\mathcal{G}^\omega=\langle T^\omega_1, H^\omega\rangle$, this is the space group $P622$.  Then $\mathcal{G}^\omega$ preserves $\Gamma^P$, and $\Theta^P_1=\Gamma^P/T^\omega_1$ is a connected graph in $T^3\cong E^3/T^\omega_1$. It has two vertices and three edges, hence its Euler characteristic $\chi(\Theta^\omega_1)=-1$ and its genus is $2$.  Clearly $G^\omega_1=\mathcal{G}^\omega/T^\omega_1$ acts on $T^3\cong E^3/T^\omega_1$ preserving $\Theta^P_1$, and $G^\omega_1\cong H^\omega$ has order $12$. Hence when we choose an equivariant regular neighbourhood $N(\Theta^P_1)$ of $\Theta^P_1$, we get an extendable action of order $12$ on $\Sigma_2\cong\partial N(\Theta^P_1)$.

Notice that when $m=n^2, 3n^2, n\in \mathbb{Z}_+$, $T^\omega_m$ is a normal subgroup of $\mathcal{G}^\omega$. Similarly the graph $\Theta^P_m=\Gamma^P/T^\omega_m$ in $T^3\cong E^3/T^\omega_m$ is connected, and $G^\omega_m=\mathcal{G}^\omega/T^\omega_m$ acts on $T^3\cong E^3/T^\omega_m$ preserving $\Theta^P_m$.
\begin{align*}
\chi(\Theta^P_m)&=\chi(\Theta^P_1)\cdot Vol(E^3/T^\omega_m)/Vol(E^3/T^\omega_1)=-m\\
|G^\omega_m|&=|G^\omega_1|\cdot Vol(E^3/T^\omega_m)/Vol(E^3/T^\omega_1)=12m
\end{align*}
Hence the genus of $\Theta^P_m$ is $m+1$. Then we can get an extendable action of order $12m$ on $\Sigma_{m+1}\cong\partial N(\Theta^P_m)$.
\end{example}

\section{Minimal surfaces, space groups, proof of the main result} \label{Sec of minimal surface}

In this section we will finish the proof  of Theorem \ref{sharp upper bound}. We need some results about triply periodic minimal surfaces, space groups, and a lemma given below.

There are three classical triply periodic minimal surfaces that admit high symmetries:  Schwarz's P surface, Schwarz's D \cite{Schw} and Schoen's gyroid surface illustrated in Figure~\ref{fig:MinSurfaces}~\cite{Br}. Denote them by $S_P$, $S_D$ and $S_G$. The fundamental 
region for $S_D$ is a cube of volume 2, we redraw $S_D$ in the fundamental region $[0,1]\times [0, 2]\times [0,1]$ we used.

\begin{figure}[h]
\begin{center}
$S_P$ \includegraphics[height=4.2cm]{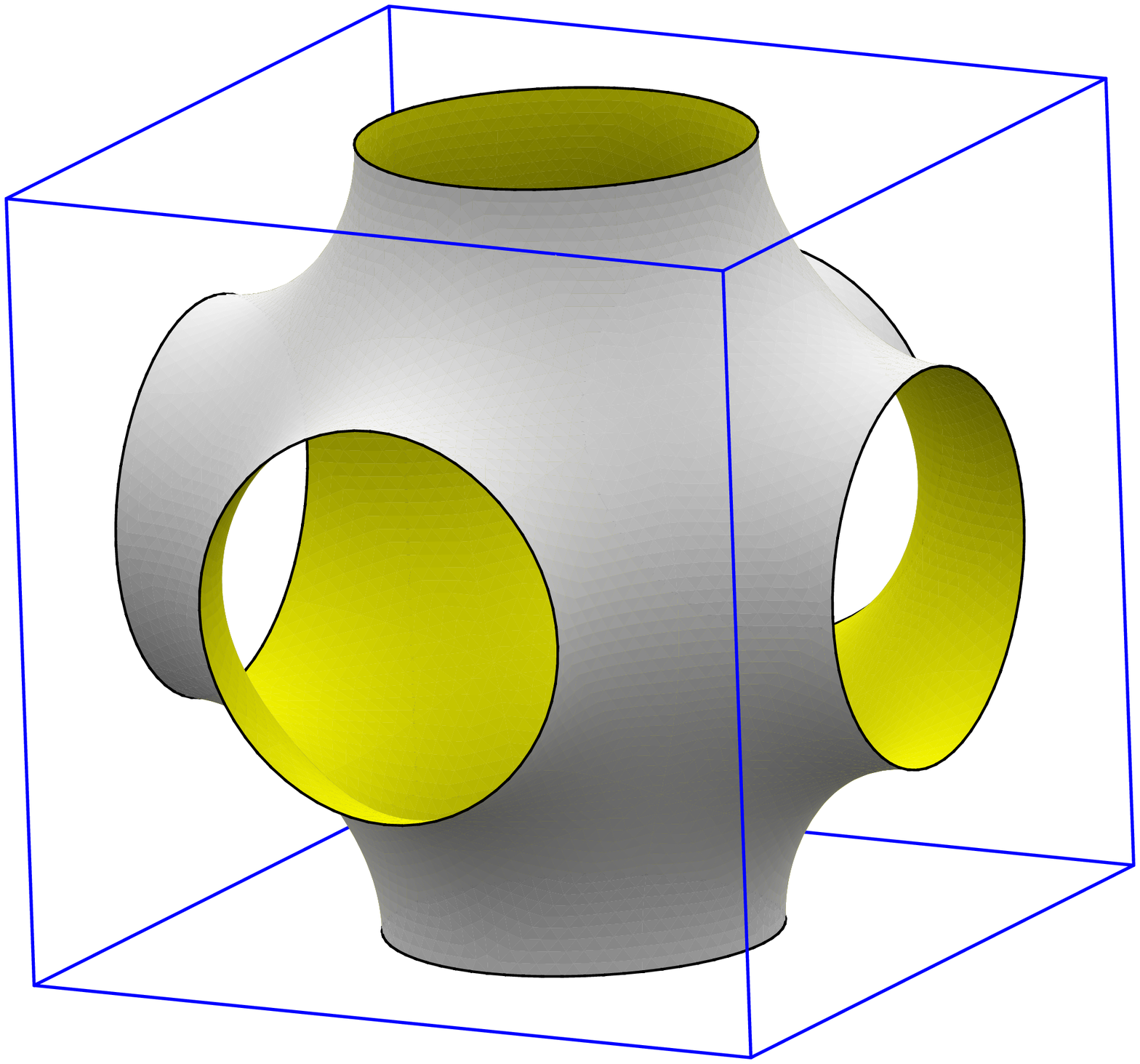}  
$S_D$ \includegraphics[height=4cm]{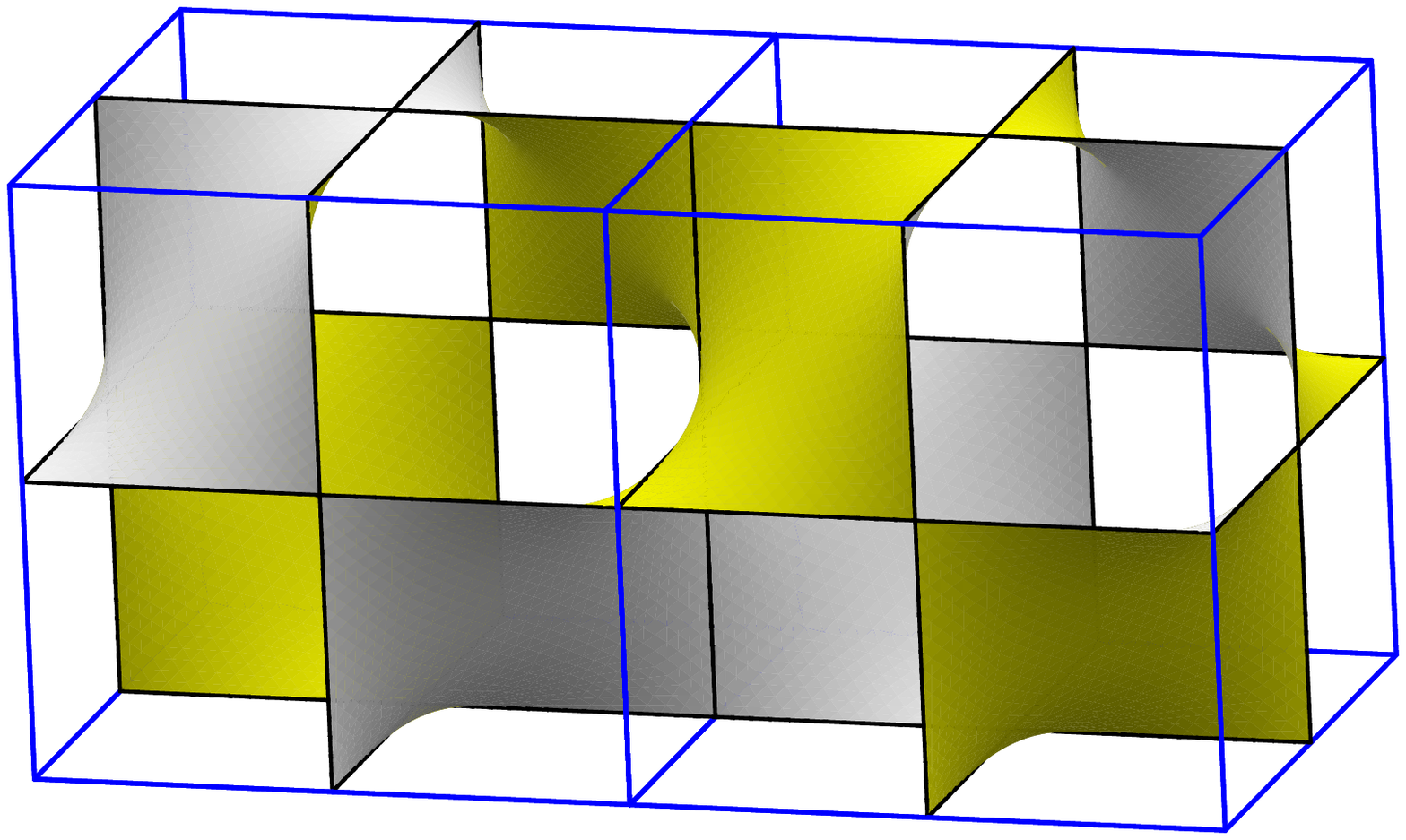}  
$S_G$ \includegraphics[height=5.3cm]{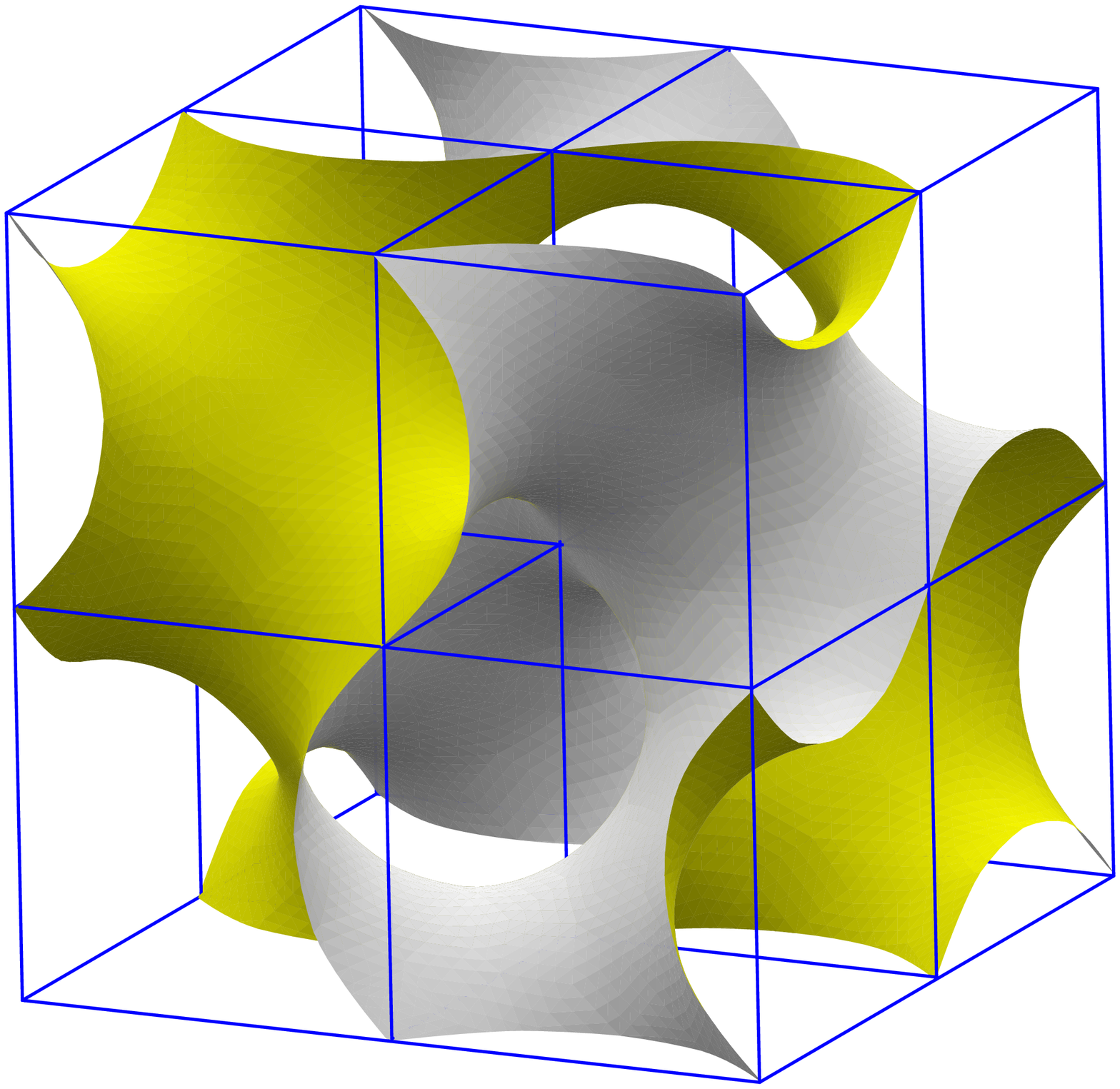}
\end{center}
\caption{Translational unit cells of the minimal surfaces $S_P$, and $S_D$ match the translational unit cells of their associated graphs, $\Gamma^1$, $\Gamma^2$, and are displayed with the same viewing direction.  The picture of $S_G$ above shows the cube $[0,2] \times [0,2] \times [0,2]$, twice the region of that used to depict $\Gamma^4$, but depicted with the same viewing direction. }\label{fig:MinSurfaces}
\end{figure}

The following results are known (and can be checked): 

(a) Boundaries of equivariant regular neighbourhoods of $\Gamma^1$, $\Gamma^2$ and $\Gamma^4$ in Example \ref{Ex of scP}, \ref{Ex of scD} and \ref{Ex of shG} can be realized by $S_P$, $S_D$ and $S_G$ respectively.

(b) $S_P$, $S_D$ and $S_G$ are unknotted in $T^3$. Actually, any genus $g>1$ minimal orientable closed surface in $T^3$ must be unknotted,\cite{Me}.

The notations of space groups below come from \cite{Ha}.

(c) $\mathcal{G}^1$, $\mathcal{G}^2$, $\mathcal{G}^4$, $\mathcal{G}^{1,2}$ and $\mathcal{G}^{2,2}$ in Example \ref{Ex of scP}, \ref{Ex of scD}, \ref{Ex of shG}, \ref{Ex of KS from scP} and \ref{Ex of KS from scD} are the space groups $[P432]$, $[F4_132]$, $[I4_132]$, $[I432]$ and $[P4_232]$ respectively.

(d) $S_P$, $S_D$, $S_G$ are preserved by space groups $[Im\bar{3}m]$, $[Pn\bar{3}m]$, $[Ia\bar{3}d]$ respectively.

(e) We have the  following index two subgroup sequences:
\begin{itemize}
\item $[P432]\subset[I432]\subset[Im\bar{3}m]$
\item $[F4_132]\subset[P4_232]\subset[Pn\bar{3}m]$
\item $[I4_132]\subset[Ia\bar{3}d]$
\end{itemize}

\begin{lemma}\label{Lem of knotted}
Let $T$ be a lattice in $E^3$, $p: E^3\rightarrow T^3 = E^3/T$ be the covering map. $\Sigma_g$ is an embedded surface in $T^3$. If $p^{-1}(\Sigma_g)$ is not connected, then $\Sigma_g$ is knotted.
\end{lemma}

\begin{proof} Note first, from the definition of unknotted embedding $\Sigma_g \subset T^3$, the induced map on the fundamental groups must be surjective.
Let $F$ be a connected component of $p^{-1}(\Sigma_g)$.  Let $St(F)$ be its stable subgroup in $T$, i.e., the lattice translations that preserve $F$. 
If $p^{-1}(\Sigma_g)$ has more than one component then $St(F) \neq T$, hence $\pi_1(\Sigma_g) \rightarrow \pi_1(T^3)$ is not surjective.
Hence $\Sigma_g$ is knotted in $T^3$.
\end{proof}


\begin{proof}[Proof of Theorem \ref{sharp upper bound}]
By the above results (a) and (b), the surfaces in $T^3$ described in \S3 are unknotted. Since $\Gamma^{1,2}$, $\Gamma^{2,2}$, $\Gamma^{4,4}$, $\Gamma^{4,8}$, $\Gamma'^4$ and $\Gamma^P$ are disconnected, by Lemma \ref{Lem of knotted}, the surfaces in $T^3$ described in \S4 are knotted. Then by combining the Examples in \S3 and \S4, we finish the proof of Theorem \ref{sharp upper bound} (1).

Now we are going to prove Theorem \ref{sharp upper bound} (2) and (3) based on the examples in \S3 and  \S4, and the results (c), (d) and (e) above.

For (2). Recall $\mathcal{G}^1=[P432]$. Example 3.3 told us that when $m=n^3, 2n^3, 4n^3$ we have the action $\mathcal{G}^1/T_m$  on $\Gamma^1/T_m$, or equivalently on $S_P/T_m$ by (a), of order $24m$. By (d) $[Im\bar{3}m]$ preserves $S_P$. 
By  (e),  we have an  order $96m$ extendable $[Im\bar{3}m]/T_m$-action on $S_P/T_m\cong \Sigma_{2m+1}$ realizing the upper bound of $E(\Sigma_g)$, where $g=2m+1$. 

The group $[Im\bar{3}m]/T_m$ contains two elements $h_1$ and $h_2$ satisfying: $h_1$ preserves the orientation of $S_P/T_m$ and reverses the orientation of $E^3/T_m$, $h_2$ reverses the orientation of $S_P/T_m$ and preserves the orientation of $E^3/T_m$.  Then choosing the index two subgroups of $[Im\bar{3}m]$,  $\langle[P432],h_1\rangle$ and  $\langle[P432],h_2\rangle$,  we get an extendable action on $\Sigma_{2m+1}$, realizing the upper bound of $E_+(\Sigma_g)$ and $E^+(\Sigma_g)$ respectively.  Now (2) is proved.

For (3).  Recall $\mathcal{G}^2=[F4_132]$. Example 3.4 told us that when $m=n^3, 4n^3, 16n^3$ we have the action $\mathcal{G}^2/T_{2m}$  on $\Gamma^2/T_{2m}$, or equivalently on $S_D/T_{2m}$ by (a), of order $24m$.   

Similarly we also have an order $96m$ extendable $[Pn\bar{3}m]/T_{2m}$-action on $S_D/T_{2m}\cong \Sigma_{2m+1}$ realizing the upper bound of $E(\Sigma_g)$, where $g=2m+1$. 

Furthermore if $m=n^3, 4n^3, n\in \mathbb{Z}_+, n$ odd,  $[Pn\bar{3}m]/T_{2m}$ contains a translation $h$ reversing an orientation of $S_D/T_{2m}$ and preserving an orientation of $E^3/T_{2m}$ (see Example 4.2 for details, or better to see $S_D$ in Figure 12). Modulo this translation we can get an order $48m$ extendable $[Pn\bar{3}m]/T_m$-action on $S_D/T_m\cong \Pi_{2m+2}$, realizing the upper bound $24(g-2)$ of $E(\Pi_g)$, where
$g=2m+2$. Then choosing the index two subgroup $[P4_232]=\mathcal{G}^{2,2}$, we can get an extendable action on $\Pi_{2m+2}$, realizing the upper bound of $E^+(\Pi_g)$. Now (3) is proved.
\end{proof}

\begin{remark} In the above discussion for (3), the regular neighbourhood $N(S_D/T_m)$ of $S_D/T_m$ is homeomorphic to a twisted $[-1,1]$-bundle over $S_D/T_m$. Its complement in $E^3/T_m$ is essentially the regular neighbourhood $N(\Theta^{2,2}_m)$ of $\Theta^{2,2}_m$ in Example \ref{Ex of KS from scD}. Similarly for $n\in \mathbb{Z}_+, n$ odd, and $m = n^3 / 2$, we can get an order $48m$ extendable $[Im\bar{3}m]/T_{n^3/2}$-action on $S_P/T_{n^3/2}\cong \Pi_{2n^3+2}$, realizing the upper bound of $E(\Pi_g)$. Choosing the index two subgroup $[I432]$, we can get an extendable action on $\Pi_{2n^3+2}$, realizing the upper bound of $E^+(\Pi_g)$. The regular neighbourhood $N(S_P/T_{n^3/2})$ is homeomorphic to a twisted $[-1,1]$-bundle over $S_P/T_{n^3/2}$. Its complement in $E^3/T_{n^3/2}$ is essentially the regular neighbourhood $N(\Theta^{1,2}_{n^3/2})$ in Example \ref{Ex of KS from scP}.
There is no similar construction for a regular neighbourhood of the Gyroid minimal surface because there is no translation that reverses the orientation of $S_G/T_{4}$; the handlebodies on each side of $S_G$ are mirror images of one another.  
\end{remark}
\bibliographystyle{amsalpha}

\end{document}